\documentclass[a4paper]{article}

\usepackage{a4wide}
\usepackage[svgnames]{xcolor}
\usepackage[utf8]{inputenc}
\usepackage{amsmath}
\usepackage{amssymb}
\usepackage{amsthm}
\usepackage{tikz-cd}
\usepackage{graphicx}
\usepackage{subfig}
\usepackage{hyperref}

\DeclareMathOperator{\id}{id}

\DeclareMathOperator{\cano}{cano}

\theoremstyle{definition}
\newtheorem{definition}{Definition}
\newtheorem{proposition}[definition]{Proposition}
\newtheorem{example}[definition]{Example}
\newtheorem{lemma}[definition]{Lemma}
\newtheorem{theorem}[definition]{Theorem}
\newtheorem{corollary}[definition]{Corollary}
\newtheorem{remark}[definition]{Remark}

\author{Jean-Christophe Aval\footnote{Univ. Bordeaux, Bordeaux INP, CNRS, LaBRI, UMR5800, F-33400 Talence, France},\hspace{0.2cm} Théo Karaboghossian\footnotemark[1]\hspace{0.2cm} and\hspace{0.2cm} Adrian Tanasa\footnotemark[1] \footnote{IUF Paris, France, EU} \footnote{H. Hulubei Nat. Inst. Phys. Nucl. Engineering Magurele, Romania, EU} \\
{\small aval@labri.fr, theo.karaboghossian@u-bordeaux.fr, ntanasa@u-bordeaux.fr}}
\title{The Hopf monoid of hypergraphs and its sub-monoids: basic invariant and reciprocity theorem}
\date{\today}

\begin{document}

\maketitle

\begin{abstract}
In arXiv:1709.07504 Aguiar and Ardila give a Hopf monoid structure on hypergraphs as well as a general construction of polynomial invariants on Hopf monoids. Using these results, we define in this paper a new polynomial invariant on hypergraphs. We give a combinatorial interpretation of this invariant on negative integers which leads to a reciprocity theorem on hypergraphs. Finally, we use this invariant to recover well-known invariants on other combinatorial objects (graphs, simplicial complexes, building sets, etc) as well as the associated reciprocity theorems.
\end{abstract}

\section{Introduction}

In combinatorics, Hopf structures give an algebraic framework to deal with operations of merging (product) and splitting (co-product) combinatorial objects. The notion of Hopf algebra is well known and used in combinatorics for over 30 years, and has proved its great strength in answering various questions (see for example \cite{HAC}). More recently, Aguiar and Mahajan defined a notion of Hopf monoid \cite{am2},\cite{am} akin to the notion of Hopf algebra and built on Joyal’s theory of species \cite{joy}.
Such as in the case of Hopf algebras, a useful application of Hopf monoids is to define and compute polynomial invariants (see \cite{ABS}, \cite{BB}, \cite{DKT} or \cite{KMT} for various examples), as was put to light by the recent and extensive paper of Aguiar and Ardila \cite{AA}. In particular they give a theorem to generate various polynomial invariants and use it to recover the chromatic polynomial of graphs, the Billera-Jia-Reiner polynomial of matroids and the strict order polynomial of posets. Furthermore they also give a way to compute these polynomial invariants on negative integers hence also recovering the different reciprocity theorems associated to these combinatorial objects.

In this paper, we apply Aguiar and Ardila's theorem to the Hopf monoid of hypergraphs defined in \cite{AA}. This Hopf structure is different than the one defined and studied  in \cite{HypB} (the respective co-products are different).
We obtain a combinatorial description for the (basic) invariant $\chi_I(H)(n)$ in terms of colorings of hypergraphs (Theorem \ref{chi}). We then use another approach (rather technical) than the method of \cite{AA} to get a reciprocity theorem for hypergraphs (Theorem \ref{chi-n}).
Finally, we use these results to obtain polynomial invariants on sub-monoids of the Hopf monoid of hypergraphs.

This paper is an extended version of the extended abstract for FPSAC 2019 \cite{fpsac}.

\section{Definitions and reminders}
\subsection{Hopf monoids}

We recall here basic definitions on Hopf monoids. The interested reader may refer to \cite{am} and to \cite{AA} for more information on this topic. In this paper, $\Bbbk$ is a field and all vector spaces are over $\Bbbk$.

\begin{definition}
A \textit{vector species} $P$ consists of the following data:
\begin{itemize}
\item for each finite set $I$, a vector space $P[I]$,
\item for each bijection of finite sets $\sigma: I\rightarrow J$, a linear map $P[\sigma]:P[I]\rightarrow P[J]$. These maps should be such that $P[\sigma\circ\tau] = P[\sigma]\circ P[\tau]$ and $P[\id] = \id$.
\end{itemize}

A \textit{sub-species} of a vector species $P$ is a vector species $Q$ such that for each finite set $I$, $Q[I]$ is a sub-space of $P[I]$ and for each bijection of finite sets $\sigma: I\rightarrow J$, $Q[\sigma] = P[\sigma]_{|Q[I]}$.

For $P$ and $Q$ two vector species, a \textit{morphism} $f: P\rightarrow Q$ between $P$ and $Q$ is a collection of linear maps $f_I : P[I] \rightarrow Q[I]$ satisfying the naturality axiom: for each bijection $\sigma: I\rightarrow J$, $f_J\circ P[\sigma] = Q[\sigma]\circ f_I$. 
\end{definition}

\begin{definition}
A \textit{connected Hopf monoid in vector species} is a vector species $M$ with $M[\emptyset] = \Bbbk$ that is equipped with product and co-product linear maps
\begin{displaymath}
\mu_{S,T}: M[S]\otimes M[T] \rightarrow M[S\sqcup T], \qquad \Delta_{S,T}: M[S\sqcup T] \rightarrow M[S]\otimes M[T],
\end{displaymath}
with $S$ and $T$ disjoint sets, and subject to the following axioms.
\begin{itemize}
\item \textit{Naturality}. For each pair of disjoint sets $S$, $T$, each bijection $\sigma$ with domain $S\sqcup T$, we have $M[\sigma]\circ\mu_{S,T} = \mu_{\sigma(S),\sigma(T)}\circ M[\sigma_{|S}]\otimes M[\sigma_{|T}]$ and $M[\sigma_{|S}]\otimes M[\sigma_{|T}]\circ\Delta_{S,T} = \Delta_{\sigma(S),\sigma(T)}\circ M[\sigma]$.
\item \textit{Unitality}. For each set $I$, $\mu_{I,\emptyset}$, $\mu_{\emptyset,I}$, $\Delta_{I,\emptyset}$ and $\Delta_{\emptyset,I}$ are given by the canonical isomorphisms $M[I]\otimes\Bbbk\cong\Bbbk\cong\Bbbk\otimes M[I]$.
\item \textit{Associativity}. For each triplet of pairwise disjoint sets $R$,$S$, $T$, we have: $\mu_{R,S\sqcup T}\circ\id\otimes\mu_{S,T} = \mu_{R\sqcup S, T}\circ \mu_{R,S}\otimes\id$.
\item \textit{Co-associativity}. For each triplet of pairwise disjoint sets $R$,$S$, $T$, we have: $\Delta_{R,S}\otimes\id\circ\Delta_{R\sqcup S, T} = \id\otimes\Delta_{S,T}\circ\Delta_{R,S\sqcup T}$.
\item \textit{Compatibility}. For each pair of disjoint sets $A$, $B$, each pair of disjoint sets $C$, $D$ ,we have the following commutative diagram, where $\tau$ maps $x\otimes y$ to $y\otimes x$:
\begin{center}
\begin{tikzcd}
\\
P[S]\otimes P[T] \arrow[d, "\Delta_{A,B}\otimes\Delta_{C,D}" swap] \arrow[r, "\mu_{S,T}"] & P[I] \arrow[r, "\Delta_{S',T'}"] & P[S']\otimes P[T']  \\
P[A]\otimes P[B]\otimes P[C]\otimes P[D] \arrow[rr, "\id\otimes\tau\otimes\id" swap] & & P[A]\otimes P[C]\otimes P[B]\otimes P[D] \arrow[u, "\mu_{A,C}\otimes\mu_{B,D}" swap]
\\
\end{tikzcd}
\end{center}
\end{itemize}

A \textit{sub-monoid} of a Hopf monoid $M$ is a sub-species of $M$ stable under the product and co-product maps. 

The \textit{co-opposite} Hopf monoid $M^{cop}$ of $M$ is the Hopf monoid with opposite co-product: $\Delta^{M^{cop}}_{S,T}=\Delta^{M}_{T,S}$.

A \textit{morphism of Hopf monoids} in vector species is a morphism of vector species which preserves the products, co-products (compatibility axiom) and the unity (unitality axiom).
\end{definition}

\begin{remark}
Readers more familiar with Hopf algebras can see connected Hopf monoids in vector species as a way to refine the coproduct of connected graded Hopf algebras. In fact there exists a functor $F$ called the \textit{Fock functor} from the category of Hopf monoids into the category of graded Hopf algebras. This functor is such that for a Hopf monoid $M$, the elements of size $n$ of $F(M)$ are the elements of $M[[n]]$ quotiented by the action of the symmetric group: $F(M)_n = M[[n]]_{S_n}$. The coproduct on $F(M)_n$ is then of the form $\Delta_n = f\circ\sum_{S\sqcup T = [n]}\Delta_{S,T}\circ i$ with $i$ and $f$ well chosen maps.
\end{remark}

We will use the term Hopf monoid for connected Hopf monoid in vector species. A sub-monoid of a Hopf monoid $M$ is itself a Hopf monoid when equipped with the product and co-product maps of $M$. We consider this to always be the case.

A \textit{decomposition} of a finite set $I$ is a sequence of pairwise disjoint subsets $S=(S_1,\dots,S_l)$ such that $I=\sqcup_{i=1}^lS_i$. A \textit{composition} of a finite set $I$ is a decomposition of $I$ without empty parts. We will write $S \vdash I$ for $S$ a decomposition of $I$, $S\vDash I$ if $S$ is a composition, $l(S) = l$ the \textit{length} of a decomposition and $|S| = |I|$ the number of elements in the decomposition.

\begin{definition}
Let $M$ a be a Hopf monoid. The \textit{antipode} of $M$ is the collection of maps $\text{S}_I : M[I]\rightarrow M[I]$ given by $\text{S}_{\emptyset} = \id$ and $$ \text{S}_I = \sum_{\substack{(S_1,\dots S_k)\vDash I\\ k\geq 1}}(-1)^{k} \mu_{S_1,\dots,S_k}\circ\Delta_{S_1,\dots S_k},$$ for any non empty finite set $I$.
\end{definition}

This expression of the antipode is known as \textit{Takeuchi's formula}.

\begin{definition}
A \textit{character} on a Hopf monoid $M$ is a collection of linear maps $\zeta_I:M[I]\rightarrow \Bbbk$ subject to the following axioms.
\begin{itemize}
\item \textit{Naturality}. For each bijection $\sigma: I\rightarrow J$, we have $\zeta_J\circ M[\sigma] = \zeta_I$.
\item \textit{Multiplicativity}. For each disjoint sets $S$, $T$, we have $\zeta_{S\sqcup T}\circ\mu_{S,T} = \mu_{\Bbbk}\circ\zeta_S\otimes\zeta_T$.
\item \textit{Unitality}. $\zeta_{\emptyset}(1) = 1$.
\end{itemize}
\end{definition}

Let us recall from \cite{AA} the results which we will use in the sequel.

\begin{definition}
\label{definv}
Let $M$ be a Hopf monoid and $\zeta$ a character on $M$. For $x\in M[I]$ and $n$ an integer, we define:
\begin{displaymath}
\chi_I(x)(n)= \sum_{(S_1, \dots S_n) \vdash I} \zeta_{S_1}\otimes\dots\otimes\zeta_{S_n}\circ\Delta_{S_1,\dots S_n}(x).
\end{displaymath}
\end{definition}

\begin{theorem}[Proposition 16.1 and Proposition 16.2 in \cite{AA}]
\label{inv}
Let $M$ be a Hopf monoid and $\zeta$ a character on $M$ and let $\chi$ be the collection of maps of Definition \ref{definv}. Then $\chi_I(x)$ is a polynomial invariant in $n$ such that:
\begin{enumerate}
\item $\chi_I(x)(1) = \zeta(x)$,
\item $\chi_{\emptyset} = 1$ and $\chi_{S\sqcup T}(\mu(x\otimes y)) = \chi_S(x)\chi_T(y)$,
\item $\chi_I(x)(-n) = \chi_I(\text{S}_I(x))(n)$.
\end{enumerate}
\end{theorem}

Let $M$ be a Hopf monoid. For $I$ a set and $x\in M[I]$ we call $x$ \textit{discrete} if $I=\{i_1,\dots, i_{|I|}\}$ and $x=\mu_{\{i_1\},\dots,\{i_{|I|}\}}x_1\otimes\dots\otimes x_{|I|}$ for $x_j\in M[\{i_j\}]$. Then the maps that send discrete elements onto 1 and other elements onto 0 give us a Hopf monoid character. Following the terminology introduced in Section 17 of \cite{AA}, we call the \textit{basic invariant of $M$} the polynomial invariant of Definition \ref{definv} with this character. We denote $\chi^M$ this polynomial or just $\chi$ when this is clear from the context.

\begin{example}
As shown in Subsection \ref{graphs}, there exists a Hopf monoid structure on graphs whose basic invariant is the chromatic polynomial.
\end{example}

\begin{proposition}[Proposition 16.3 in \cite{AA}]
\label{invmorph}
Let $M$ and $N$ be two Hopf monoids, $\zeta^M$ and $\zeta^N$ characters on $M$ and $N$ and $f: M\rightarrow N$ a Hopf monoid morphism such that for every $I$: $$\zeta^N_I\circ f_I = \zeta^M_I.$$ Denote by $\chi(\zeta^M)$ and $\chi(\zeta^N)$ the polynomial invariants of Definition \ref{definv} with $M$ and $\zeta^M$ and $N$ and $\zeta^N$. For every $I$, one then has: $$ \chi(\zeta^N)_I\circ f_I = \chi(\zeta^M)_I.$$
\end{proposition}

In particular, since Hopf monoid morphisms conserve discrete elements, for $f: M\rightarrow N$ a Hopf monoid morphism and $I$ a set, we have $\chi^N_I\circ f_I = \chi^M_I$.

\subsection{A useful combinatorial identity}

We recall here a classical result of combinatorics and a direct corollary which will be useful in the following section. We only give a sketch of the proofs.

In all the following, given an integer $n$ we will denote by $[n]$ the set $\{1,\dots, n\}$.

\begin{proposition}
Let $n$ and $m$ be two integers. The number of surjections $S_{n,m}$ from $[m]$ to $[n]$ is given by: $$S_{n,m} = \sum_{k=0}^n(-1)^{n-k}\binom{n}{k}k^m.$$ 
\end{proposition}

\begin{proof}
This formula can be obtained by the inclusion-exclusion principle.
\end{proof}

\begin{corollary}
\label{idcombi}
For $n$ and $m$ two integers such that $m<n$, and $P$ a polynomial of degree at most $m$, we have: $$\sum_{k=0}^n(-1)^{n-k}\binom{n}{k}P(k) = 0.$$
\end{corollary}

\begin{proof}
The statement above is a direct consequence of the fact that $S_{n,m} = 0$ for $n<m$.
\end{proof}

\section{Basic invariant of hypergraphs}

In all of the following, $I$ always denotes a finite set.

Our goal is to express the basic invariant of the Hopf monoid of hypergraphs defined in Section 20 of \cite{AA}. More specifically we intend to obtain a combinatorial interpretation of $\chi_I(x)(n)$ and $\chi_I(x)(-n)$.

In this context, a \textit{hypergraph over $I$} is a collection of (possibly repeated) subsets of $I$, which we call \textit{edges}\footnote{in some references, the terms \textit{hyperedge} or \textit{multiedge} is used.}. The elements of $I$ are then called \textit{vertices} of $H$ and $HG[I]$ denotes the free vector space of hypergraphs over $I$. Note that two hypergraphs over different sets can never be equal, e.g $\{\{1,2,3\},\{2,3,4\}\}\in HG[[4]]$ is not the same as $\{\{1,2,3\},\{2,3,4\}\}\in HG[[4]\cup\{a,b\}]$. This is illustrated in Figure \ref{same-edges}.

\begin{figure}[htbp]
\begin{center}
\includegraphics[scale=1.5]{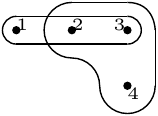}
\hspace{5cm}
\includegraphics[scale=1.5]{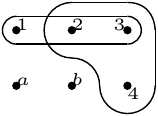}
\caption{Two hypergraphs with the same edges but over different sets.}
\label{same-edges}
\end{center}
\end{figure}

The product and co-product for $I=S\sqcup T$ are given by,
\begin{align*}
\mu_{S,T}: HG[S]\otimes HG[T] &\rightarrow HG[I] & \Delta_{S,T}: HG[I] &\rightarrow HG[S]\otimes HG[T] \\
H_1\otimes H_2 &\mapsto H_1\sqcup H_2 & H &\mapsto H_{|S}\otimes H_{/S}
\end{align*}
where $H_{|S} = \{e\in H\, |\, e\subseteq S\}$ is the \textit{restriction of $H$ to $S$} and $H_{/S} = \{e\cap T\, |\, e \nsubseteq S\}$ is the \textit{contraction of $S$ from $H$}. The discrete hypergraphs are then the hypergraphs with edges of cardinality at most 1.

\begin{example}
For $I=[5]$, $S=\{1,2,5\}$ and $T=\{3,4\}$, we have:
\begin{center}
\begin{tikzcd}[column sep=huge]
\raisebox{-21pt}{\includegraphics[scale=1.5]{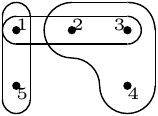}} \arrow[r, "\Delta_{S,T}"] & \raisebox{-10pt}{\includegraphics[scale=1.5]{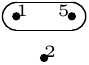}}\hspace{0.2cm} \otimes \hspace{0.2cm}\raisebox{-10pt}{\includegraphics[scale=1.5]{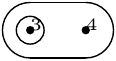}}
\end{tikzcd}
\end{center}
\end{example}


In \cite{AA}, Aguiar and Ardila propose a method to obtain a combinatorial interpretation of any polynomial invariant given in Definition \ref{definv} on negative integers, assuming that we have an interpretation of it on positive integers. Their method consists in using a cancellation-free grouping-free formula for the antipode and point 3 of Theorem \ref{inv}.
Here we use a different approach: we express the polynomial dependency of $\chi_I(x)(n)$ in $n$, which we then use to calculate $\chi_I(x)(-n)$ and interpret the resulting formula.

Let us begin by giving a proposition which is needed to show the polynomial dependency of $\chi_I(x)(n)$ in $n$.
For $t\in\mathbb{N}^*$ and a sequence of positive integers $p_1, p_2,\dots, p_t$, we define $F_{p_1,\dots, p_t}$ as a function over the integers given by, for $n\in\mathbb{N}$:$$F_{p_1,\dots, p_t}(n) = \sum_{0\leq k_1 <\dots < k_t\leq n-1} k_1^{p_1}\cdots k_t^{p_t}.$$
Note that if $t>n$, then $F_{p_1,\dots, p_t}(n) = \sum_{\emptyset}\dots=0$.

\begin{proposition}
Let $p_1, p_2,\dots, p_t$ be integers and define $d_k = \sum_{i=1}^k p_i + k$ for $1\leq k \leq t$. Then $F_{p_1,\dots, p_t}$ is a polynomial of degree $d_t$ whose constant coefficient is null and the $(d_t - i)$-th, (for $i < d_t$) coefficient is given by
\begin{displaymath}
\sum_{j_{t-1} = 0}^{\min(j_t, d_{t-1}-1)}\sum_{j_{t-2}}^{\min(j_{t-1}, d_{t-2}-1)}\dots\sum_{j_1 = 0}^{\min(j_2,d_1 - 1)}\prod_{k=1}^t \binom{d_k - j_{k-1}}{j_k - j_{k-1}}\frac{B_{j_k - j_{k-1}}}{d_k - j_{k-1}},
\end{displaymath}
where $j_t = i$ and $j_0 = 0$, and the $B_j$ numbers are the Bernoulli numbers with the convention $B_1 = -1/2$.
\end{proposition}

\begin{proof}
We show this by induction on $t$. For $t=1$ the expression of the coefficients gives us the well-known identity $F_p(n) = \sum_{i=0}^p \binom{p+1}{i}\frac{B_i}{p+1}n^{p+1-i}$. Hence the result is true for $t=1$. Suppose now the result is true for $t\geq 1$ and let $p_1, p_2,\dots, p_{t+1}$ be $t+1$ integers. Denote by $a_i$ the $d_t -i$ coefficient of $F_{p_1,\dots, p_t}(n)$. We then have:
\begin{align*}
F_{p_1,\dots, p_{t+1}}(n) &= \sum_{0\leq k_1 <\dots < k_{t+1}\leq n-1} k_1^{p_1}\cdots k_{t+1}^{p_{t+1}} = \sum_{k_{t+1}=0}^{n-1}k_{t+1}^{p_{t+1}}\sum_{0\leq k_1 <\dots < k_t\leq k_{t+1}-1}  k_1^{p_1}\cdots k_t^{p_t}\\
&= \sum_{k = 0}^{n-1} k^{p_{t+1}}F_{p_1,\dots, p_t}(k) \\
&= \sum_{k = 0}^{n-1} k^{p_{t+1}}\sum_{j = 0}^{d_t -1} a_j k^{d_t-j} \\
&= \sum_{j = 0}^{d_t - 1} a_j\sum_{k = 0}^{n-1} k^{p_{t+1} + d_t -j} = \sum_{j = 0}^{d_t - 1} a_j\sum_{k = 0}^{n-1} k^{d_{t+1} -1 - j} \\
&= \sum_{j = 0}^{d_t - 1} a_j F_{d_{t+1} - 1 - j}(n) \\
&= \sum_{j = 0}^{d_t - 1} a_j\sum_{i = 0}^{d_{t+1}-1-j}\binom{d_{t+1}-j}{i}\frac{B_i}{d_{t+1} - j} n^{d_{t+1} - j -i} \\
&= \sum_{j = 0}^{d_t - 1}\sum_{i = 0}^{d_{t+1}-1-j} a_j\binom{d_{t+1}-j}{i}\frac{B_i}{d_{t+1} - j} n^{d_{t+1} - j -i} \\
&= \sum_{j = 0}^{d_t - 1}\sum_{i = j}^{d_{t+1}-1} a_j\binom{d_{t+1}-j}{i-j}\frac{B_{i-j}}{d_{t+1} - j} n^{d_{t+1} - i} \\
&= \sum_{i = 0}^{d_{t+1} - 1}\left(\sum_{j = 0}^{\min(i,d_t - 1)} a_j\binom{d_{t+1}-j}{i-j}\frac{B_{i-j}}{d_{t+1}-j}\right) n^{d_{t+1} - i}.
\end{align*} 
This concludes this proof.
\end{proof}

Before stating our results on $\chi_I(H)(n)$ we need to introduce some definitions. There exists a canonical bijection between decompositions and functions with co-domain of the form $[n]$. In the sequel, we will want to seamlessly pass from one notion to the other. We hence give a few explanations on this bijection. Given an integer $n$, the canonical bijection between decompositions of $I$ of size $n$ and functions from $I$ to $[n]$ is given by:
\begin{align*}
b_{I,n}: \{f:I\rightarrow [n]\} &\rightarrow \{P\vdash I\,|\, l(P)=n\} \\
f &\mapsto (f^{-1}(1),\dots,f^{-1}(n)).
\end{align*}
If it is clear from the context what are $I$ and $n$, we will write $b$ instead of $b_{I,n}$. If $P$ is a partition we will also refer to $b^{-1}(P)$ by $P$ so that instead of writing ``i such that $v\in P_i$'' and ``i and j such that $v\in P_i$, $v'\in P_j$ and $i<j$ '' we can just write $P(v)$ and $P(v)<P(v')$. Similarly, if $P$ is a function we will refer to $b(P)$ by $P$ so that $P_i = P^{-1}(i)$. Also remark that $b_{I,n}$ induces a bijection between compositions of $I$ of size $n$ and surjections from $I$ to $[n]$.

\begin{definition}
Let $H$ be a hypergraph over $I$ and $n$ be an integer. A \textit{coloring of $H$ with $[n]$} is a function from $I$ to $[n]$ (or a decomposition of $I$ of length $n$ from what precedes this) and in this context the elements of $[n]$ are called \textit{colors}.

Let $S\vdash I$ be a coloring of $H$. For $v\in e\in H$, we say that $v$ is a \textit{maximal vertex of $e$ (for $S$)} if $v$ is of maximal color in $e$ and we call the \textit{maximal color of $e$ (for $S$)} the color of a maximal vertex of $e$. We say that a vertex $v$ is a \textit{maximal vertex (for $S$)} if it is a maximal vertex of an edge.

If $J\subseteq I$ is a subset of vertices, the \textit{order of appearance of $J$ (for $S$)} is the composition $\cano(S_{|J})$ where $S_{|J} = (S_1\cap J,\dots, S_{l(S)}\cap J)$. The map $\cano$ sends any decomposition to the composition obtained by dropping the empty parts.

\end{definition}

\begin{example}
\label{hgcolo}
We represent the coloring of a hypergraph on $I=\{a,b,c,d,e,f\}$ with $\{${\color{DarkGreen}1$(\bullet)$},{\color{DarkBlue}2$(\times)$},{\color{DarkMagenta}3$(\square)$},{\color{DarkRed}4$(\blacksquare)$}$\}$:
\begin{center}
\includegraphics[scale=1.6]{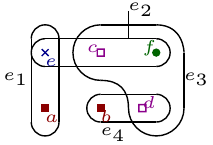}
\end{center}
The maximal vertex of $e_1$ is $a$ and the maximal vertices of $e_3$ are $c$ and $d$. The maximal color of $e_2$ is {\color{DarkMagenta}{3}}. The order of appearance of $\{a,c,d,e\}$ is $(\{e\},\{c,d\},\{a\})$.
\end{example}

\begin{definition}
Let $H$ be a hypergraph over $I$. An \textit{orientation} of $H$ is a function $f$ from $H$ to $I$ such that $f(e)\in e$ for every edge $e$. A \textit{directed cycle} in an orientation $f$ of $H$ is a sequence of distinct edges $e_1,\dots, e_k$ such that $f(e_1)\in e_2\setminus f(e_2),\dots, f(e_k)\in e_1\setminus f(e_1)$. An orientation is \textit{acyclic} if it does not have any directed cycle. Let $\mathcal{A}_H$ be the set of acyclic orientations of $H$.

An orientation $f$ of $H$ and a coloring $S$ of $H$ with $[n]$ are said to be \textit{compatible} if $S(f(e)) = \max(S(e))$ for every $e\in H$. They are said to be \textit{strictly compatible} if $f(e)$ is the unique maximal vertex of $e$.
\end{definition}

\begin{example}
The coloring given in Example \ref{hgcolo} has two compatible acyclic orientations: both send $e_1$ on $a$, $e_2$ on $c$ and $e_4$ on $b$, but one sends $e_3$ on $c$ and the other $e_3$ on $d$.

For the color set $\{${\color{DarkBlue}1$(\bullet)$},{\color{DarkRed}2$(\times)$}$\}$, the following coloring has 4 compatible orientations but only two are acyclic.
\begin{center}
\includegraphics[scale=1.6]{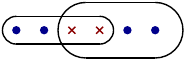}
\end{center}
\end{example}

\begin{theorem}
\label{chi}
Let $I$ be a set and $H\in HG[I]$ a hypergraph over $I$. Then $\chi_I(H)(n)$ is the number of colorings of $H$ with $[n]$ such that every edge has only one maximal vertex. This is also the number of strictly compatible pairs of acyclic orientations and colorings with $[n]$.
Furthermore, defining $P_{H,f} = \{ P \vDash f(H)\, |\,  v\in e\setminus f(e) \Rightarrow P(v)< P(f(e))\}$, for every $f\in \mathcal{A}_H$, 
we have that $$\chi_I(H)(n) = n^{|J_H|}\sum_{f\in \mathcal{A}_H}\sum_{P\in P_{H,f}} F_{p_1,\dots, p_{l(P)}}(n),$$ where $J_H\subseteq I$ is the set of isolated vertices of $H$ (i.e vertices not in an edge) and for every $P\in P_{H,f}$, $p_i =|\tilde{P}_i|$ and $\tilde{P}_i = \left(\bigcup_{e\in f^{-1}(P_i)}e\right) \cap f(H)^c\bigcap_{j<i} \tilde{P}_j^c$.
\end{theorem}

\begin{proof}
For $S$ a decomposition of $I$ of size $n$, let $H_1\in HG[S_1],\dots,H_n\in HG[S_n]$ be hypergraphs such that $H_1\otimes\dots\otimes H_n = \Delta_{S_1,\dots, S_n}(H)$. Let $S$ be a decomposition of $I$ of size $n$. Let $e$ be an edge. We then have the equivalence:
\begin{align*}
e\in H_i &\iff e\cap S_i \not = \emptyset \wedge\forall j>i, e\cap S_j = \emptyset \\
&\iff \text{$e\cap S_i$ is the set of maximal vertices of $e$}
\end{align*}
Hence, we have that 
\begin{align*}
\zeta_{S_1}\otimes\dots\otimes\zeta_{S_n}\circ\Delta_{S_1,\dots, S_n}(H) = 1 &\iff \forall e\in H, e\in H_i \Rightarrow |e\cap S_i| = 1\\
&\iff\text{each edge has only one maximal vertex.}
\end{align*}
The equivalence between the colorings such that every edge has only one maximal vertex and the strictly compatible pairs of acyclic orientations and colorings is given by the bijection $S\mapsto (e\mapsto v_e,S)$, where $v_e$ is the unique vertex in $e$ such that $S(v_e) = \max(S(e))$.

The term $n^{|J_H|}$ in the formula is trivially obtained, in the following we hence consider that $H$ has no isolated vertices.

Informally, the formula can be obtained by the following reasoning. To choose a coloring such that every edge has only one maximal vertex, one can proceed in the following way:
\begin{enumerate}
\item choose the maximal vertex of each edge ($f\in \mathcal{A}_H$),
\item choose in which order those vertices appear ($P\in P_{H,f}$),
\item choose the color of those vertices ($k_1+1,\dots,k_{l(P)}+1$), (and notice that the set of such choices is empty if $l(P)> n$, which allows us not to add this non polynomial dependency in $n$ at the previous choice),
\item choose the colors of the yet uncolored vertices which are in the same edge than a vertex of minimal color in $f(H)$ ($k_1^{|\tilde{P}_1|}$); then those in the same edge than a vertex of second minimal color in $f(H)$ ($k_2^{|\tilde{P}_2|}$), etc.
\end{enumerate}

More formally, we show that there exists a bijection between the set of colorings such that every edge has only one maximal vertex and the set $$\bigsqcup_{f\in \mathcal{A}_H}\bigsqcup_{P\in P_{H,f}}\bigsqcup_{0\leq k_1 < k_2 < \dots < k_{l(P)}\leq n-1} \prod_{1\leq i \leq l(P)}[k_i]^{\tilde{P}_i},$$ where $[k_i]^{\tilde{P}_i}$ is the set of maps from $\tilde{P}_i$ to $[k_i]$. Let $g$ be a coloring of interest and define:
\begin{itemize}
\item $f:e \mapsto v\in e$ such that $g(v) = \max(g(e))$,
\item $P = h\circ g(f(H))$ where $h$ is the increasing bijection from $g(f(H))$ to $[|f(H)|]$,
\item $\tilde{P}_i = \left(\bigcup_{e\in f^{-1}(P_i)}e\right) \cap f(H)^c\bigcap_{j<i} \tilde{P}_j^c$ for $1\leq i\leq l(P)$,
\item $k_i = g(P_i) - 1$ for $1\leq i \leq l(P)$.
\end{itemize}
The function $f$ not being in $\mathcal{A}_H$ would imply that there exists a vertex $v$ such that $g(v)<g(v)$. This is not possible, hence $f\in \mathcal{A}_H$. We also have that $P\in P_{H,f}$ because by definition of $g$, $v\in e\setminus f(e)$ implies $g(v) < g(f(e))$ and $h$ is increasing. It is also clear that $0\leq k_1< \dots < k_{l(P)} \leq n-1$. The image of $g$ is then $(g_{|\tilde{P}_1},\dots,g_{|\tilde{P}_{l(P)}})$ which is in $\prod_{1\leq i \leq l(P)}[k_i]^{\tilde{P}_i}$ since for every $v\in \tilde{P}_i$ we must have $g(v) < g(P_i)$ by definition. Let us now consider $f\in \mathcal{A}_H$, $P\in P_{H,f}$, $1\leq k_1 < \dots < k_{l(P)}$ and $(g_1,\dots g_{l(P)})\in \prod_{1\leq i \leq l(P)}[k_i]^{\tilde{P}_i}$. Let $h$ be the increasing bijection from $[l(P)]$ to $\{k_1+1,\dots, k_{l(P)}+1\}$ and define $g: I \rightarrow [n]$ by $g_{|\tilde{P}_i} = g_i$ and $g_{|f(H)} = h\circ P$ (it is sufficient since $(\tilde{P}_1,\dots \tilde{P}_{l(P)},f(H))$ is a partition of $I$). Let us show that $g$ is a coloring of interest. Let be $v\in e\setminus f(e)$,
\begin{itemize}
\item if $v\in f(H)$ then $P(v)<P(f(e)$) by definition and so $g(v) < g(f(e))$, since $h$ is increasing, 
\item if $v\not\in f(H)$ then $v\in \tilde{P}_i$ with $i\leq P(f(e))$ and so $g(v) = g_i(v) \leq k_i < k_i + 1 \leq k_{P(f(e))} + 1 = g(f(e))$.
\end{itemize}
We conclude the proof by remarking that the two defined transformations are inverse functions.
\end{proof}

\begin{example}
The coloring given in Example \ref{hgcolo} is not counted in $\chi_I(H)(4)$ since $e_3$ has two maximal vertices. However by changing the color of $d$ to {\color{DarkBlue}2} we do obtain a coloring where every edge has only one maximal vertex.

Let $H$ be the hypergraph $\{\{1,2,3\},\{2,3,4\}\}\in HG[[4]]$ represented in Figure \ref{same-edges}. Then we have $\chi_{[4]}(H)(n)= n^4 - \frac{8}{3}n^3+\frac{5}{2}n^2-\frac{5}{6}n$ and we do verify that, for example, $\chi_{[4]}(H)(2)=3$.
\end{example}

We are now interested in the value of $(-1)^{|I|}\chi_I(H)(-n)$. Let us first state two lemmas. The first lemma justifies the use of the $F_p$ (with $p$ a finite sequence of integers) polynomials to express the basic invariant: they have a good expression on negative integers. The second lemma is a result which can be interpreted on graphs and partitions as we do, but also on posets and linear extensions. It is the crux of the proof of Theorem \ref{chi-n}.

Given an integer $n$, a sequence of positive integers $p=p_1,\dots, p_t$ is said to be a \textit{decomposition} of $n$ if $n=\sum_{i=1}^t p_i$. We denote this by $p\vdash n$. If $p=(p_{1,1},\dots,p_{1,k_1},p_{2,1},\dots,p_{2,k_2},\dots,p_{l,k_l})$ and $q=(q_1,\dots, q_l)$ are two decompositions of the same integer, we say that $q$ \textit{coarsens} $p$ or $p$ \textit{refines} $q$ and write $p\prec q$ if $(p_{i,1},\dots,p_{i,k_i})$ is a decomposition of $q_i$ for $1\leq i\leq l$.

\begin{example}
The sequences $p_1=(4,2,3,3)$, $p_2=(3,1,2,2,1,1,1,1)$ and $p_3 = (4,3,1,1,3)$ are three decompositions of $12$ such that $p_2$ refines $p_1$ and there is no relation of refinement and coarsening between $p_3$ and the two other sequences.
\end{example}

\begin{lemma}
\label{f-n}
Let $p$ be a sequence of positive integers of length $t$. Then $$F_{p}(-n) = (-1)^{d_t}\sum_{p\prec q} F_q(n+1).$$
\end{lemma}

\begin{proof}
Remark that $\sum_{p\prec q} F_q(n+1)$ can also be written as $\sum_{0\leq k_1 \leq\dots\leq k_t\leq n}k_1^{p_1}\cdots k_t^{p_t}$. We now proceed by induction on $t$. For $t=1$, we have 
\begin{align*}
F_p(-n) &= \sum_{i=0}^p \binom{p+1}{i}\frac{B_i}{p+1}(-n)^{p+1-i} \\
&= \frac{(-1)^{p+1}}{p+1}n^{p+1} -\frac{1}{2}(-1)^pn^p + (-1)^{p+1}\sum_{i=2}^p \binom{p+1}{i}\frac{B_i}{p+1}n^{p+1-i} \\
&= (-1)^{p+1}\left(\frac{1}{p+1}n^{p+1} +\frac{1}{2}n^p + \sum_{i=2}^p \binom{p+1}{i}\frac{B_i}{p+1}n^{p+1-i}\right) \\
&= (-1)^{p+1}(F_p(n) + n^p) = (-1)^{p+1}F_p(n+1),
\end{align*}
where the second equality comes from the fact that $B_i =0$ when $i$ is an odd number different from 1. Suppose now our proposition is true up to $t$. In the proof of Proposition 2 we showed that $F_{p_1,\dots, p_{t+1}}(n) = \sum_{j = 0}^{d_t - 1} a_j F_{d_{t+1} - 1 - j}(n)$ where $a_j$ is the $d_t-j$ coefficient of $F_{p_1,\dots, p_{t}}(n)$. This gives
\begin{align*}
F_{p_1,\dots, p_{t+1}}(-n) &= \sum_{j = 0}^{d_t - 1} a_j (-1)^{d_{t+1} - j}\sum_{k=0}^{n}k^{d_{t+1} - 1 - j} = -\sum_{j = 0}^{d_t - 1} a_j\sum_{k = 0}^{n} (-k)^{p_{t+1} + d_t -j} \\
&= -\sum_{k = 0}^{n} (-k)^{p_{t+1}}\sum_{j = 0}^{d_t -1} a_j (-k)^{d_t-j} = (-1)^{p_{t+1}+1}\sum_{k = 0}^n k^{p_{t+1}} F_{p_1,\dots,p_t}(-k) \\
&= (-1)^{p_{t+1}+1}\sum_{k_{t+1} = 0}^n k_{t+1}^{p_{t+1}} (-1)^{d_t}\sum_{0\leq k_1 \leq\dots\leq k_t\leq k_{t+1}}k_1^{p_1}\cdots k_t^{p_t} \\
&= (-1)^{d_{t+1}}\sum_{0\leq k_1 \leq\dots\leq k_{t+1} \leq n}k_1^{p_1}\cdots k_{t+1}^{p_{t+1}} \\
&= (-1)^{d_{t+1}}\sum_{p\prec q} F_q(n+1),
\end{align*}
where the fifth equality is our induction hypothesis.
\end{proof}


\begin{definition}
Let $I$ and $J$ be two disjoint sets and $P = (P_1,\dots,P_l)\vDash I$ and $Q = (Q_1,\dots, Q_k) \vDash J$ be two compositions. The \textit{product} of $P$ and $Q$ is the composition $P\cdot Q = (P_1,\dots,P_l,Q_1,\dots Q_k)$. The \textit{shuffle product} of $P$ and $Q$ is the set $sh(P,Q) = \{R\vDash I\sqcup J\,|\, P=\cano (R_{|I}), Q=\cano(R_{|J})\}$.

Let $P'=(P_{1,1},\dots,P_{1,k_1},P_{2,1},\dots,P_{2,k_2},\dots,P_{l,k_l})$ be another composition of $I$. We say that $P'$ \textit{refines} $P$ and write $P'\prec P$ if $P_i = \bigcup_{j=1}^{k_i}P_{i,j}$ for $1\leq i\leq l$.
\end{definition}

Recall that compositions can be seen as surjections and that for a decomposition $P\vDash I$ and an element $v\in I$, we denote by $P(v)$ the index $i$ such that $v\in P_i$.

\begin{lemma}
\label{compsum}
Let $I$ be a set and $P\vDash I$ a composition of $I$. We have the identity: $$\sum_{Q\prec P} (-1)^{l(Q)} = (-1)^{|P|}.$$

Furthermore let $G$ be a directed acyclic graph on $I$ and consider the \textit{constrained set} \linebreak[4] $C(G,P)=\{Q\prec P\,|\, \forall (v,v')\in G,Q(v)< Q(v')\}$. We have the more general identity: $$\sum_{Q\in C(G,P)} (-1)^{l(Q)} = \left\{\begin{array}{cl}
0  & \text{if there exists $(v,v')\in G$ such that $P(v')<P(v)$},    \\ 
(-1)^{|P|}  & \text{if not.}\end{array}\right.$$
\end{lemma}

\begin{proof}
Since $\sum_{Q\prec P} (-1)^{l(Q)} = \prod_{i=1}^{l(P)}\sum_{Q\vDash P_i} (-1)^{l(Q)}$ we only need to show that $\sum_{Q\vDash I} (-1)^{l(Q)} = (-1)^{|I|}$ to prove the first identity. Since the compositions of $I$ of size $n$ and the surjections from $I$ to $[n]$ are in bijection, we have that:
\begin{align*}
\sum_{Q\vDash I} (-1)^{l(Q)} &= \sum_{n=1}^{|I|} (-1)^n S_{|I|,n} = \sum_{n=1}^{|I|} (-1)^n \sum_{k=1}^n (-1)^{n-k}\binom{n}{k}k^{|I|} \\
&= \sum_{k=1}^{|I|}(-1)^k\left(\sum_{n=k}^{|I|}\binom{n}{k}\right)k^{|I|} = \sum_{k=1}^{|I|}(-1)^k\binom{|I|+1}{k+1}k^{|I|} \\
&= (-1)^{|I|}\sum_{k=0}^{|I|-1}(-1)^{k}\binom{|I|+1}{k}(|I|-k)^{|I|} \\
&= (-1)^{|I|}(1 + \sum_{k=0}^{|I|+1}(-1)^{k}\binom{|I|+1}{k}(|I|-k)^{|I|}) \\
&= (-1)^{|I|}.
\end{align*}
Note that the last equality is a direct consequence of Corollary \ref{idcombi}.

To show the second identity first remark that the case where the sum is null is straightforward: if there exists $(v,v')\in G$ such that $P(v')<P(v)$, then $C(G,P) = \emptyset$ and so the sum is null. From now on we only consider non empty summation sets. In this case we have that $\sum_{Q\in C(G,P)} (-1)^{l(Q)} = \prod_{i=1}^{l(P)}\sum_{Q\in C(G\cap {P_i}^2,(P_i))} (-1)^{l(Q)}$ and we only need to show that \linebreak[4] $\sum_{P\in C(G)} (-1)^{l(P)} = (-1)^{|I|}$ where $C(G)=C(G,(I))$. Let $S(G)$ denote the sum $\sum_{P\in C(G)} (-1)^{l(P)}$ from now on.

If $G$ is not connected let $I=J\sqcup K$ and $G=H\sqcup H'$ where $V(H) = J$ and $V(H') = K$. Let $P\in C(H)$ and $Q\in C(H')$ and suppose without loss of generality that $m = l(Q) < l(P) = M$. To choose $R$ in $sh(P,Q)$ we can first choose its length; then which indices are going to have a part of $Q$; and then which indices among them are also going to have a part of $P$. This leads to:
\begin{align*}
\sum_{R\in sh(P,Q)} (-1)^{l(R)} &= \sum_{k=M}^{m+M}(-1)^{k}\binom{k}{m}\binom{m}{M-(k-m)} \\
&= \sum_{k=0}^{m} (-1)^{k+M}\binom{M+k}{m}\binom{m}{k-m} \\
&= (-1)^{M}\sum_{k=0}^{m} (-1)^{m-k}\binom{M+m-k}{m}\binom{m}{k} \\
&= \frac{(-1)^{M}}{m!}\sum_{k=0}^{m} (-1)^{m-k}\binom{m}{k}\frac{(m+M-k)!}{(M-k)!} \\
&= \frac{(-1)^{M}}{m!}\sum_{k=0}^{m} (-1)^{m-k}\binom{m}{k}(-k)^m \\
&= \frac{(-1)^{m+M}}{m!}\sum_{k=0}^{m} (-1)^{m-k}\binom{m}{k}k^m \\
&= \frac{(-1)^{m+M}}{m!}S_{m,m} = (-1)^{m+M} = (-1)^{l(P)+l(Q)},
\end{align*}
where the fifth equality follows from Corollary \ref{idcombi}. This shows that $S(G)$ is multiplicative (with the product being the disjoint union) and so we can restrict ourselves to showing that $S(G) = (-1)^{|I|}$ for $G$ a connected graph. We will do this by induction on the number of edges of $G$.

Suppose now that $G$ is connected. If $G$ has no edges then $G$ is reduced to a single vertex and the result is trivial. Thus let be $(v,v')\in G$. We say that $(v,v')$ is superfluous if there exists $v_0$, $v_1$, ..., $v_{k+1}\in I$ such that $v=v_0$, $v'=v_{k+1}$ and $(v_i,v_{i+1})\in G$ for all $i\in [k]$. If $(v,v')$ is superfluous then $C(G) = C(G(v,v'))$ and so $S(G) = S(G\setminus(v,v')) = (-1)^{|I|}$ by induction. Otherwise we have $C(G\setminus(v,v')) = C(G) + C(t_{(v,v')}(G)) + C(G\setminus(v,v'))\cap\{P\vDash I\,|\, P(v) = P(v')\}$, where $t_{(v,v')}$ sends $G$ on $G\setminus(v,v')\cup(v',v)$. By induction, we know that $S(G\setminus(v,v')) = (-1)^{|I|}$ and since $C(G\setminus(v,v'))\cap\{P\vDash I\,|\, P(v) = P(v')\} = C\left(G\cap(I/v')^2\cup\bigcup_{(w,v')\in G\setminus(v,v')} (w,v)\cup\bigcup_{(v',w)\in G} (v,w)\right)$, we also have by induction that $\sum_{P\in C(G\setminus(v,v'))\cap\{P\vDash I\,|\, P(v) = P(v')\}} (-1)^{l(P)} = (-1)^{|I|-1}$. Hence, we have the equivalence $S(G) = (-1)^{|I|} \iff S(t_{(v,v')}(G)) = (-1)^{|I|}$.

Let $e_1, \dots, e_k$ be a sequence of edges such that for every $i$, $G_i = t_{e_i}\circ\dots\circ t_{e_1}(G)$ does not have a directed cycle. Then we have that $S(G) = (-1)^{|I|}$ if and only if $S(G_k) = (-1)^{|I|}$. If $G$ has a cycle then we can find a sequence such that $G_k$ has a superfluous edge and hence $S(G_k) = (-1)^{|I|}$. If $G$ does not have any cycle then every sequence of edges satisfies the condition ``$G_i$ does not have a directed cycle'' and so $S(G) = (-1)^{I}$ as long as there exists a directed graph $G'$ with the same underlying non-oriented graph than $G$ such that $S(G') = (-1)^{|I|}$. Given a non-oriented connected graph $H$, we can always find a directed graph $G$ on it with only one vertex $v$ such that for every $w\in V(G)$, $(w,v)\not \in G$. Then we have that $C(G) = (\{v\})\cdot C(G\cap (V(G)-v)^2)$ which gives us $S(G) = -S(G\cap (V(G)-v)^2) = (-1)^{|V(G)|}$ by induction. This concludes the proof.
\end{proof}

We can now state the main result of this section which is a direct generalization to hypergraphs of the reciprocity theorem of Stanley on graphs \cite{stan}.

\begin{theorem}[Reciprocity theorem on hypergraphs]
\label{chi-n}
Let $I$ be a set and $H\in HG[I]$ a hypergraph over $I$. Then $(-1)^{|I|}\chi_I(H)(-n)$ is the number of compatible pairs of acyclic orientations and colorings with $[n]$ of $H$. In particular, $(-1)^{|I|}\chi_I(H)(-1) = |\mathcal{A}_H|$ is the number of acyclic orientations of $H$.
\end{theorem}

\begin{proof}
From Proposition \ref{chi} and Lemma \ref{f-n} we have that $$\chi_I(H)(-n) = (-n)^{|J_H|}\sum_{f\in \mathcal{A}_H}\sum_{P\in P_{H,f}}(-1)^{\sum_{i=1}^{l(P)}p_i + l(P)}\sum_{(p_1, \dots, p_{l(P)}) \prec q}F_q(n+1).$$
Remark that:
\begin{itemize}
\item $\sum_{i=1}^{l(P)}p_i = |I\setminus J_H| -|f(H)|$ (since $(\tilde{P}_1,\dots \tilde{P}_{l(P)}, f(H))$ is a partition of $I\setminus J_H$),
\item $\phi :\{Q\vDash f(H)\,|\, P\prec Q\} \rightarrow \{q\vdash(|I\setminus J_H| -|f(H)|)\,|\, (p_1,\dots, p_{l(P)}) \prec q\}$ $Q\mapsto (|\tilde{Q}_1|,\dots, |\tilde{Q}_{l(Q)}|)$ is a bijection (where $\tilde{Q}_i$ is defined in the same way that $\tilde{P}_i$ in Theorem \ref{chi}).
\end{itemize}
This leads to:
\begin{align*}
(-1)^{|I|}\chi_I(H)(-n) &= n^{|J_H|}\sum_{f\in \mathcal{A}_H}(-1)^{|f(H)|}\sum_{P\in P_{H,f}}(-1)^{l(P)}\sum_{P\prec Q} F_{\phi(Q)}(n+1)\\
&= n^{|J_H|}\sum_{f\in \mathcal{A}_H}(-1)^{|f(H)|}\sum_{Q\vDash f(H)}\left(\sum_{\substack{P\prec Q\\P\in P_{H,f}}} (-1)^{l(P)}\right)F_{\phi(Q)}(n+1).
\end{align*}
By definition of $\mathcal{A}_H$, $G = \{(v,f(e))\,|\, v\in e\setminus f(e)\}$ is a directed acyclic graph on $f(H)$. Hence, remarking that $\{ P\prec Q\,|\, P\in P_{H,f}\} = C(G,Q)$, Lemma \ref{compsum} leads to:
\begin{align*}
(-1)^{|I|}\chi_I(H)(-n) &= n^{|J_H|}\sum_{f\in \mathcal{A}_H}(-1)^{|f(H)|}\sum_{\substack{P\vDash f(H)\\ P(v)\leq P(v') \forall(v,v')\in G}} (-1)^{|f(H)|}F_{\phi(P)}(n+1) \\
&= n^{|J_H|}\sum_{f\in \mathcal{A}_H}\sum_{P\in P_{H,f}'} F_{\phi(P)}(n+1) \\
&= n^{|J_H|}\sum_{f\in \mathcal{A}_H}\sum_{P\in P_{H,f}'}F_{p_1,\dots, p_{l(P)}}(n+1),
\end{align*}
where $P_{H,f}' = \{ P\vDash f(H)\,|\, P(v\in e\setminus f(e) \leq P(f(e)\}$.
To conclude, we now need to show that the set of compatible pairs (acyclic orientation, coloring with $n$) is in bijection with $$\bigsqcup_{f\in \mathcal{A}_H}\bigsqcup_{P\in P_{H,f}'}\bigsqcup_{0\leq k_1 < \dots < k_{l(P)}\leq n}[k_1]^{Q_1}\times\dots\times [k_{l(P)}]^{Q_{l(P)}}.$$ This can be done in a way analogous to the one used in the proof of Theorem \ref{chi}, the only difference being that we choose (with the same terms used in the proof) $g(P_i) = k_i$ instead of $g(P_i) = k_i+1$.
\end{proof}

\begin{example}
For any $I$, any $H\in HG[I]$ and $n$ a positive integer, we have $\chi_I(H)(n) \leq (-1)^{|I|}\chi_I(H)(-n)$. This comes from the fact that any strictly compatible pair is compatible. This is observed for $H=\{\{1,2,3\},\{2,3,4\}\}\in HG[[4]]$: $$\chi_{[4]}(H)(n)= n^4 - \frac{8}{3}n^3+\frac{5}{2}n^2-\frac{5}{6}n < n^4 + \frac{8}{3}n^3+\frac{5}{2}n^2+\frac{5}{6}n = (-1)^4\chi_{[4]}(H)(-n).$$
We also verify that $H$ does have $\chi_{[4]}(H)(-1) = 7$ acyclic orientations ($3\times 3$ orientations minus the two cyclic orientations).
\end{example}

\section{Application to other Hopf monoids}

In this section we use Theorem \ref{chi} and Theorem \ref{chi-n} to obtain a combinatorial interpretation of the basic invariants for the Hopf monoids presented in Sections 20 to 25 of \cite{AA}.

The general method to do this will be to use the fact that these Hopf monoids can be seen as sub-monoids of the Hopf monoid of simple hypergraphs, and then present an interpretation of what is an acyclic orientation on these particular Hopf monoids.

The result from subsection \ref{sshg} is new. The result from subsection \ref{graphs} already appears in \cite{AA}. The results of subsections \ref{sc} to \ref{paths} are new, but they can be derived from more general results in previous papers (details are provided at the beginning of each subsection). Note that however, we present here a uniform approach to obtain these results.

In all the following, we denote by $\chi$ the basic invariant of the Hopf monoid of hypergraphs.

\subsection{Simple hypergraphs}
\label{sshg}
A hypergraph is \textit{simple} if it has no repeated edges. The vector species $SHG$ of simple hypergraphs is not stable by the contraction defined on hypergraphs but it still admits a Hopf monoid structure. The product and co-product are given by, for $I=S\sqcup T$:
\begin{align*}
\mu_{S,T}: SHG[S]\otimes SHG[T] &\rightarrow SHG[I] & \Delta_{S,T}: SHG[I] &\rightarrow SHG[S]\otimes SHG[T] \\
H_1\otimes H_2 &\mapsto H_1\sqcup H_2 & H &\mapsto H_{|S}\otimes H_{/S},
\end{align*}
where $H_{|S} = \{e\in H\, |\, e\subseteq S\}$ and $H_{/S} = \{e\cap T\, |\, e \nsubseteq S\}\cup\{\emptyset\}$ but this time without repetition, i.e $H_{/S}$ can also be defined as $\{B\subseteq\,|\, \exists A\subseteq S, A\sqcup B \in H\}$. A discrete simple hypergraph is then a simple hypergraph with edges of cardinality at most one.

\begin{proposition}
\label{shg}
$\chi^{SHG}$ is the restriction of $\chi$ to the vector species of simple hypergraphs.
\end{proposition}

\begin{proof}
Let $s: HG\rightarrow SHG$ be the Hopf monoid morphism which removes any repetition of edges and let $H$ be a simple hypergraph over $I$. Considering $SHG$ as a sub-species of $HG$ and $s$ as a morphism of vector species we have: $\chi^{HG}_I(H) = \chi^{HG}_I(s(H))$. Then using the fact that $s$ is a Hopf monoid morphism stable on the sub-species of discrete elements we have that: $\chi^{HG}_I(s(H))=\chi^{SHG}_I(H)$. This concludes the proof.
\end{proof}

\subsection{Graphs}
\label{graphs}
The result of this subsection has already been given in Section 18 of \cite{AA}, but we give it here as a consequence of our result in the previous section. 

A \textit{graph} can be seen as a hypergraph whose edges are all of cardinality 2. As for the vector species of simple hypergraphs, the vector species $G$ of graphs is not stable by the contraction defined on hypergraphs, but it still admits a Hopf monoid structure. The product and co-product are given by, for $I=S\sqcup T$:
\begin{align*}
\mu_{S,T}: G[S]\otimes G[T] &\rightarrow G[I] & \Delta_{S,T}: G[I] &\rightarrow G[S]\otimes G[T] \\
g_1\otimes g_2 &\mapsto g_1\sqcup g_2 & g &\mapsto g_{|S}\otimes g_{/S},
\end{align*}
where $g_{|S}$ is the sub-graph of $g$ induced by $S$ and $g_{/S} = g_{|T}$. A discrete graph is then a graph with no edges.

A \textit{proper coloring} of a graph is a coloring such that no edge has its two vertices of the same color. The chromatic polynomial of a graph is the polynomial $T$ such that $T(n)$ is the number of proper colorings with $n$ colors.

\begin{corollary}[Proposition 18.1 in \cite{AA}]
The basic invariant of $G$ is the chromatic polynomial.
\end{corollary}

\begin{proof}
Let $HG_{\leq 2}$ be the Hopf monoid of hypergraphs with edges of cardinality at most 2 and let $s:HG_{\leq 2}\rightarrow G$ be the Hopf monoid morphism which removes edges of cardinality 1. Since $HG_{\leq 2}$ is a Hopf sub-monoid of $HG$ we have that $\chi^{HG_{\leq 2}}$ is the restriction of $\chi$ to $HG_{\leq 2}$. Using the same reasoning than in the proof of Proposition \ref{shg}, with $G$ and $HG_{\leq 2}$ instead of $SHG$ and $HG$, we get that $\chi^{G}$ is the restriction of $\chi$ to $G$.
Furthermore, for $g$ a graph and $S$ a coloring of $g$, we have the equivalence between ``each edge has a unique maximal vertex'' and ``$S$ is a proper coloring''. The result follows.
\end{proof}

In particular, by evaluating $\chi$ on negative integers for a graph, we recover the classical reciprocity theorem of Stanley \cite{stan}.

\subsection{Simplicial complexes}
\label{sc}
In \cite{sc} Benedetti, Hallam, and Machacek constructed a combinatorial Hopf algebra of simplicial complexes and in particular they obtained results which generalise those given in this subsection. 

An \textit{abstract simplicial complex}, or simplicial complex, on $I$ is a collection $C$ of subsets of $I$, called \textit{faces}, such that any subset of a face is a face i.e $J \in C$ and $K \subset J$ implies $K \in C$. By Proposition 21.1 of \cite{AA}, the vector species $SC$ of simplicial complexes is a sub-monoid of the Hopf monoid of simple hypergraphs.

The \textit{1-skeleton} of a simplicial complex is the graph formed by its faces of cardinality 2.

\begin{corollary}
Let $I$ be a set, $C\in SC[I]$ and $g$ its 1-skeleton. Then $\chi^{SC}_I(C)$ is the chromatic polynomial of $g$.
\end{corollary}

\begin{proof}
It is clear that any coloring of $C$ such that each edge has a unique maximal vertex induces a proper coloring of $g$. On the other hand if $J$ is a face of $C$ then it is also a clique of $g$, and so any proper coloring of $g$ must color all the vertices in $J$ in different colors. In particular there must be a unique maximal vertex in $J$.
\end{proof}

\subsection{Building sets}
Building sets and graphical building sets have been studied in a Hopf algebraic context by Gruji\'c in \cite{gr} where he gave similar results to the ones obtained in this subsection and the following one.

Building sets were independently introduced by De Concini and Procesi in \cite{DCP} and by Schmitt in \cite{Schmi}. A \textit{building set} on $I$ is collection $B$ of subsets of $I$, called \textit{connected sets}, such that if $J,K\in B$ and $J\cap K\not = \emptyset$ then $J\cup K\in B$ and for all $i\in I$, $\{i\}\in B$. By Proposition 22.3 of \cite{AA} the vector species $BS$ of building sets is a sub-monoid of the Hopf monoid of simple hypergraphs.

The maximal sets of a building set are called \textit{connected components}.


\begin{definition}
Let $B$ be a building set on $I$ with only one connected component. We define recursively a set of rooted trees which we call \textit{skeletons} of $B$.
\begin{itemize}
\item If $I=\{r\}$ is a singleton and $B=\{\{r\}\}$ the sole skeleton of $B$ is the rooted tree reduced to its root $r$.
\item If $I$ is not a singleton let $r$ be any element of $I$ and let $I_1,\dots,I_k$ be the maximal connected sets of $B$ not containing $r$. Then for each $i$ we associate to $I_i$ the building set $B_i=\{J\,|\,J\in B\text{ and }J\subseteq I_i\}$ on $I_i$. For each of these building sets, choose a skeleton $s_i$ and denote $r_i$ its root. Then the tree $s=\{(r,r_1),\dots,(r,r_k)\}\cup s_1\cup\dots\cup s_k$ with root $r$ is a skeleton of $B$.
\end{itemize}

Let $B$ a building set. Its \textit{skeletons} are the disjoint unions of skeletons of its connected component.
\end{definition}


\begin{remark}
The skeletons are exactly the $B$-$forests$ where all the vertices are singletons as defined in Definition 22.6 of \cite{AA}.
\end{remark}

A rooted forest can be seen as a forest with an orientation which sends each edge on the parent vertex. Hence, one can define compatible and strictly compatible colorings of a rooted forest. Moreover, these notions correspond to the notions of \textit{natural-$T$-partition} and \textit{strict-$T$-partition} of \cite{gr}.

\begin{corollary}
\label{corbuild}
Let $I$ be a set and $B\in BS[I]$. Then $\chi^{BS}_I(B)(n)$ is the number of strictly compatible pairs of skeletons and colorings with $[n]$ and $\chi^{BS}_I(B)(-n)$ is the number of compatible pairs of skeletons and colorings with $[n]$. In particular, $\chi^{BS}_I(B)(-1)$ is the number of skeletons of $B$.
\end{corollary}

\begin{proof}
Since $BS$ is a sub-monoid of $SHG$ we know that $\chi^{BS}$ is the restriction of $\chi$ to $BS$. Hence, we only need to show that there exists a bijection $b$ which preserves compatibility with colorings between the acyclic orientations of $B$ seen as a hypergraph and its skeletons. Furthermore, since the acyclic orientations of a disjoint union of hypergraphs are the disjoint unions of their acyclic orientations and the skeletons of a disjoint union of building sets are the disjoint unions of their skeletons, we only have to show this bijection for a building set with one connected component.

We will do this by induction on the size of $I$. If $I=\{r\}$ is a singleton and $B=\{\{r\}\}$, then the unique acyclic orientation of $B$ is $f:\{r\}\mapsto r$ and the unique skeleton of $B$ is the rooted tree with only vertex $r$. These two elements are both compatible with all the colorings hence the preserving bijection $b: f \mapsto \{r\}$.

If $I$ is not a singleton, let $K$ be the connected component of $B$.
Let $f$ be an acyclic orientation of $B$ and $r=f(K)$. Then necessarily all the connected sets of $B$ containing $r$ are also sent on $r$ by $f$ because otherwise we would have a cycle (since these connected sets are contained in $K$ by definition). Let now $I_1,\dots I_k$ be the maximum connected sets not containing $r$ and $B_1,\dots,B_k$ their associated building sets. Then $f_{|I_i}$ is an acyclic orientation of $B_i$ and $s_i = b(f_{|I_i})$ is a skeleton of $B_i$ for $1\leq i\leq k$. Let $r_i$ be the root of $s_i$ for $1\leq i\leq k$. Then, by definition of a skeleton, the tree $b(f)=\{(r,r_1),\dots,(r,r_k)\}\cup s_1\cup\dots\cup s_k$ rooted in $r$ is a skeleton of $B$.

Let now $s$ be a skeleton of $B$ and $r$ be its root. Let $B_0$ be the set of connected sets of $B$ containing $r$ and $I_0$ be the set of vertices which are in a connected set containing $r$ which is not the connected component. Let $I_1,\dots, I_k$ be the maximal connected set of $B$ not containing $r$ and $B_1,\dots,B_k$ be their associated building sets. Note that since $B$ is a building set, one has $I_i\cap I_j = \emptyset$ for $1\leq i\not = j\leq k$ (or else $I_i\subsetneq I_i\cup I_j\in B$ is not maximal).

By definition of a skeleton, $s$ has exactly $k$ sub-trees $s_1,\dots, s_k$ such that $s_i$ is a skeleton of $B_i$ for $1\leq i\leq k$.
Define $f_i=b^{-1}(s_i)$ for $1\leq i\leq k$ and $f_0$ the acyclic orientation of $B_0$ which sends every connected set of $B_0$ on $r$. Let $f$ be the orientation of $B$ which sends a connected set $K\in B_i$ on $f_i(K)$ for $i$ such that $K\in B_i$. This orientation is everywhere well defined because $B_0,B_1,\dots,B_k$ is, by definition, a partition of $B$. Suppose $f$ is not acyclic and let $e_1,\dots e_l$ be a directed cycle. Then necessarily the connected sets $e_1,\dots e_l$ can not all be in the same $B_i$ because $f_{|I_i} = f_i$ is acyclic. Without loss of generality, let $e_1\in B_{i_1}$ and $e_2\in B_{i_2}$ with $i_1\not = i_2$.
\begin{itemize}
\item Suppose $i_1 \not = 0$ and $i_2 \not = 0$. Then $f(e_1)\in e_1\cap e_2\subseteq I_{i_1}\cap I_{i_2}$, which is not possible by maximality of $I_{i_1}$.
\item Suppose $i_1 = 0$. Then $r = f(e_1) \in e_1\cap e_2\subseteq I_{i_2}$, which is not possible by definition of $I_{i_2}$.
\item Suppose $i_2 = 0$. Then by the previous point $e_3$ must also be in $B_0$. An iteration then implies that all the $e_i$ must be in $B_0$. This contradicts the hypothesis $i_1\not = i_2$.
\end{itemize}
Hence $f$ is an acyclic orientation.

The fact that in the two preceding constructions the root of the skeleton is the image of the connected component along with the induction hypothesis enable us to conclude that $b$ is bijection that preserves compatibility with colorings.
\end{proof}

\subsection{Simple graphs, ripping and sewing}
A \textit{simple} graph is a simple hypergraph that is also a graph. The vector species $W$ of simple graphs admits a Hopf monoid structure, the product and co-product are given by, for $I=S\sqcup T$:
\begin{align*}
\mu_{S,T}: W[S]\otimes W[T] &\rightarrow W[I] & \Delta_{S,T}: W[I] &\rightarrow W[S]\otimes W[T] \\
w_1\otimes w_2 &\mapsto w_1\sqcup w_2 & w &\mapsto w_{|S}\otimes w_{/S}
\end{align*}
where $w_{|S}$ is the sub-graph of $w$ induced by $S$ and $w_{/S}$ is the simple graph on $T$ with an edge between $u$ and $v$ if there is a path from $u$ to $v$ in which all the vertices which are not ends are in $S$. These two operations are respectively called \textit{ripping out $T$} and \textit{sewing through $S$}. A discrete simple graph is then a simple graph with no edges.

\begin{definition}[Definition 23.1 in \cite{AA}]
Let be $w\in W[I]$. A \textit{tube} is a subset $J\subseteq I$ such that $w_{|J}$ is connected. The set of tubes of $w$ is a building set called \textit{graphical building set of $w$} and which we denote $\text{tubes}(w)$.
\end{definition}

By Proposition 23.3 of \cite{AA} we know that $w\mapsto \text{tubes}(w)$ is a Hopf monoid morphism between $W$ and $BS$.

Given a rooted tree we call its \textit{direct sub-trees} the sub-trees with roots the children of the root.

\begin{definition}
Let be $w\in W[I]$ a connected simple graph. We define the set of \textit{partitioning trees of $w$} inductively by the following:
\begin{itemize}
\item if $I=\{v\}$, then the unique partitioning of $w$ is the graph with $\{v\}$ as only vertex,
\item else choose $v\in I$ and a partitioning tree for each connected component of $w_{|I\setminus \{v\}}$. The tree with root $v$ and direct sub-trees these partitioning trees is then a partitioning tree of $w$.
\end{itemize}
If $w$ is not connected anymore, a \textit{partitioning forest of $w$} is the disjoint union of partitioning trees of each connected component of $w$.
\end{definition}

\begin{corollary}
Let $I$ be a set and $w\in W[I]$. Then $\chi^{W}_I(w)(n)$ is the number of colorings of $w$ with $[n]$ such that every path with ends of the same color has a vertex of color strictly greater than the colors of the ends. It is also the number of strictly compatible pairs of partitioning forests and colorings with $[n]$. $\chi^{W}_I(w)(-n)$ is the number of compatible pairs of partitioning forests and colorings with $[n]$. In particular, $\chi^{W}_I(w)(-1)$ is the number of partitioning trees of $w$.
\end{corollary}

\begin{proof}
Since $w\mapsto \text{tubes}(w)$ is a Hopf monoid morphism, we know from what follows Proposition \ref{invmorph} that $\chi^W_I(w) = \chi^{BS}_I(\text{tubes}(w))$. Hence from Corollary \ref{corbuild}, $\chi^W_I(w)(n)$ is the number of strictly compatible pairs of skeletons of $\text{tubes}(w)$ and colorings with $[n]$ and $\chi^W_I(w)(-n)$ is the number of compatible pairs of skeletons of $\text{tubes}(w)$ and colorings with $[n]$.

Furthermore since $\chi^{BS}$ is the restriction of $\chi$ to $BS$ we have that $\chi^W_I(w) = \chi_I(\text{tubes}(w))$ so $\chi^{W}_I(w)(n)$ is also the number of colorings of $\text{tubes}(w)$ such that all edges have a unique maximal vertex.

Given these two facts we now only need to show the two following points:
\begin{itemize}
\item A coloring $I\rightarrow [n]$ is a coloring of $\text{tubes}(w)$ such that all edges have a unique maximal vertex if and only if it is a coloring of $w$ such that every path with ends of the same color has a vertex of color strictly greater than the colors of the ends.
\item The partitioning forests of $w$ are exactly the skeletons of $\text{tubes}(w)$.
\end{itemize}

We begin by the first assertion. Let $S$ be a coloring of $\text{tubes}(w)$ of interest and $v_1,\dots,v_k$ a path of $w$ such that $S(v_1) = S(v_k)$. Then $w_{|\{v_1,\dots, v_k\}}$ is connected and so $\{v_1,\dots, v_k\}$ is an edge of $\text{tubes}(w)$. Since $v_1$ and $v_k$ are of the same color, their color can not be the maximal color. Hence there exists an $i$ such that $S(v_i) > S(v_1) = S(v_k)$ and $S$ is a coloring of $w$ of interest. Let now $S$ be a coloring of $w$ of interest and $e$ an edge of $\text{tubes}(w)$ with two vertices $v_1$ and $v_2$ of the same color. Then, since $w_{|e}$ is connected by definition, there exists a path in $e$ from $v_1$ to $v_2$ and hence $v_3$ such that $S(v_3) > S(v_1) = S(v_2)$. Thus there can only be one vertex of maximal color in $e$ and $S$ is a coloring of $\text{tubes}(w)$ of interest.

To show that partitioning forests and skeletons are the same objects, just remark that given a vertex $v\in J$, where $w_{|J}$ is a connected component of $w$, the connected components of $w_{|J\setminus\{v\}}$ are exactly the maximal connected sets of $\text{tubes}(w)$ included in $J$ but not containing $v$.
\end{proof}

\subsection{Set partitions}
Proposition 24.4 of \cite{AA} states that there exists an isomorphism between the Hopf monoid of permutahedra and the Hopf monoid of set partitions. Furthermore, Propositions 17.3 and 17.4 of \cite{AA} give a combinatorial interpretation of the basic invariant of the Hopf monoid of generalized permutahedra $GP$. The Hopf monoid of permutahedra being a sub-monoid of a quotient of $GP$, it should be possible to deduce the result presented in this subsection from the aforementioned propositions.

A \textit{partition of $I$} is a subset of $\mathcal{P}(I)\setminus\{\emptyset\}$ such that all elements, which are called \textit{parts}, are disjoints and their union equals $I$. The vector species $\Pi$ of partitions admits a Hopf monoid structure, the product and co-product are given by, for $I=S\sqcup T$:
\begin{align*}
\mu_{S,T}: \Pi[S]\otimes \Pi[T] &\rightarrow \Pi[I] & \Delta_{S,T}: \Pi[I] &\rightarrow \Pi[S]\otimes \Pi[T] \\
\pi_1\otimes \pi_2 &\mapsto \pi_1\sqcup \pi_2 & \pi &\mapsto \pi_{|S}\otimes \pi_{|T},
\end{align*}
where for $\pi=\{\pi^1,\dots,\pi^l\}$, $\pi_{|S}$ is the partition of $S$ obtained by taking the intersection with $S$ of each part $\pi^i$ and forgetting the empty parts. A discrete partition is then a partition where all parts are singletons.

A \textit{cliquey graph} is a disjoint union of cliques. By Proposition 24.2 of \cite{AA} we know that $\pi\mapsto c(\pi)$ is a Hopf monoid from $\Pi$ to $W$, where $c(\pi)$ is the cliquey graph with a clique on each part of $\pi$.

\begin{corollary}
Let $I$ be a set and $\pi = \{\pi^1,\dots,\pi^l\}\in\Pi[I]$. Then $\chi^{\Pi}_I(\pi)(n) = \Pi_{i=1}^lp_i!\binom{n}{p_i}$ where $p_i = |\pi^i|$.
\end{corollary}

\begin{proof}
Since $\chi^{\Pi}$ is multiplicative and $\pi\mapsto c(\pi)$ is a Hopf monoid morphism, we only need to show that $\chi^{W}_I(w)(n)=|I|!\binom{n}{|I|}$ where $w$ is the clique on $I$. A coloring $S$ of $w$ is such that every path with ends of the same color has a vertex of color strictly greater than the colors of the ends if and only if all vertices are of different colors (because for each pair $v_1,v_2$ of vertices $v_1,v_2$ is a path in $w$). Hence the number of such colorings is the number of injections from $I$ to $[n]$. This concludes the proof.
\end{proof}

\subsection{Paths}
\label{paths}
As for the previous subsection, Proposition 25.7 of \cite{AA} states that the Hopf monoid of sets of paths is isomorphic to the Hopf monoid of associahedra which is a sub-monoid of a quotient of $GP$. Hence, it should also be possible to deduce the result of this subsection from \cite{AA}.

A \textit{word} on $I$ is a total ordering of $I$. The \textit{paths} on $I$ are the words on $I$ quotiented by the relation $w_1\dots w_{|I|} \sim w_{|I|}\dots w_1$. A \textit{set of paths} $\alpha$ of $I$ is a partition $(I_1,\dots, I_l)$ of $I$ with a path $s_i$ on each part $I_i$ and we will write $\alpha = s_1|\dots|s_l$. The vector species $F$ of sets of paths admits a Hopf monoid structure, the product and co-product are given by, for $I=S\sqcup T$:
\begin{align*}
\mu_{S,T}: F[S]\otimes F[T] &\rightarrow F[I] & \Delta_{S,T}: F[I] &\rightarrow F[S]\otimes F[T] \\
\alpha_1\otimes \alpha_2 &\mapsto \alpha_1\sqcup \alpha_2 & \alpha &\mapsto \alpha_{|S}\otimes \alpha_{/S}
\end{align*}
where if $\alpha = s_1|\dots|s_l$, $\alpha_{|S}= s_1\cap S|\dots| s_l\cap S$ forgetting the empty parts and $\alpha_{/S}$ is the set of paths obtained by replacing each occurrence of an element of $S$ in $\alpha$ by the separation symbol $|$. A discrete set of paths is then a set of paths where all paths have only one element.

\begin{example}
For $I=\{a,b,c,d,e,f,g\}$ and $S=\{b,c,e\}$ and $T=\{a,d,f,g\}$, we have: $$\Delta_{S,T}(bfcg|aed) = bc|e\otimes f|g|a|d.$$
\end{example}

By Proposition 25.1 of \cite{AA} we know that $\alpha \mapsto l(\alpha)$ is a morphism of Hopf monoids from $F$ to $W^{cop}$; where $l(s_1|\dots|s_l)$ is the simple graph whose connected components are the paths induced by $s_1,\dots, s_l$.

\begin{example}
For $I=\{a,b,c,d,e,f,g\}$ and $\alpha=bfcg|aed$, $l(\alpha)$ is the following graph:
\begin{center}
\includegraphics[scale=1.5]{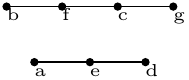}
\end{center}
\end{example}

\begin{corollary}
Let $I$ be a set and $\alpha$ be a path on $I$. Then $\chi^F_I(\alpha)(n)$ is the number of strictly compatible pairs of binary trees with $|I|$ vertices and colorings with $[n]$ and $\chi^F_I(\alpha)(-n)$ is the number of compatible pairs of binary trees with $|I|$ vertices and colorings with $[n]$. In particular $\chi^F_I(\alpha)(-1) = C_{|I|}$ where $C_n = \frac{1}{n+1}\binom{2n}{n}$ is the $n$-th Catalan number.
\end{corollary}

\begin{proof}
First remark that by definition, $\chi^{W^{cop}} = \chi^{W}$ and so $\chi^F_I(\alpha)(n) = \chi^W_I(l(\alpha))(n)$. Fix one of the two total orderings of $I$ induced by $\alpha$ so that we can consider the left and the right of a vertex $v$ of $l(\alpha)$. Then each vertex of $l(\alpha)$ is totally characterised by the number of vertices on its left (and on its right) and hence the partitioning trees of $l(\alpha)$ are exactly the binary trees with $|I|$ vertices.
\end{proof}

\bigskip
\bigskip
\noindent{\bf Concluding remarks}

Let us end this paper by presenting some perspectives for future work.
We plan to generalize the results of this paper to all characters on the Hopf monoid of hypergraphs. While Theorem \ref{chi} does generalize easily for characters with value in $\{0,1\}$, the conditions on the characters are slighlty more involved for Theorem \ref{chi-n} to hold.

Finally, an open question that appears interesting to us is to recover Theorem \ref{chi-n} using the antipode formula given in \cite{AA}. We refer the reader to \cite{HypB} where this has been done for a different Hopf structure on hypergraphs.
\bigskip

\noindent{\bf Acknowledgement}

The authors would like to thank Alexander Postnikov for pointing out to them (after the submission of this paper) a different approach to obtain the basic invariant for hypergraphs \cite{Pos1}\cite{Pos2}. The authors are also grateful to the anonymous referee for his useful comments on a preliminary version of this paper.

\bibliographystyle{plain}
\bibliography{redac}
\end{document}